\documentclass[9pt,final,twocolumn]{IEEEtran}
\long\def\comment#1{}

\usepackage{textcomp}
\usepackage{cite}
\usepackage{graphicx}
\usepackage{subfigure}
\usepackage{psfrag}
\usepackage{url}
\usepackage{stfloats}
\usepackage{amsmath}
\usepackage{amssymb,amsfonts}
\usepackage{amsthm}
\usepackage{float}
\usepackage[usenames]{color}
\usepackage{algorithm,algorithmic}

\newtheorem{assumption}{Assumption}
\newtheorem{remark}{Remark}

\newtheorem{lemma}{Lemma}
\newtheorem{theorem}{Theorem}
\newtheorem{example}{Example}

\begin{document}

\setlength{\arraycolsep}{0.3em}

\title{Distributed Aggregative Optimization over Multi-Agent Networks
\thanks{}}

\author{Xiuxian Li, Lihua Xie, and Yiguang Hong
\thanks{X. Li and L. Xie are with School of Electrical and Electronic Engineering, Nanyang Technological University, 50 Nanyang Avenue, Singapore 639798 (e-mail: xxli@ieee.org; elhxie@ntu.edu.sg).}
\thanks{Y. Hong is with Key Laboratory of Systems and Control, Institute of Systems Science, Chinese Academy of Sciences, Beijing, 100190, China (email: yghong@iss.ac.cn).}
}

\maketitle

\setcounter{equation}{0}
\setcounter{figure}{0}
\setcounter{table}{0}

\begin{abstract}
This paper proposes a new framework for distributed optimization, called distributed aggregative optimization, which allows local objective functions to be dependent not only on their own decision variables, but also on the average of summable functions of decision variables of all other agents. To handle this problem, a distributed algorithm, called distributed gradient tracking (DGT), is proposed and analyzed, where the global objective function is strongly convex, and the communication graph is balanced and strongly connected. It is shown that the algorithm can converge to the optimal variable at a linear rate. A numerical example is provided to corroborate the theoretical result.
\end{abstract}


\section{Introduction}\label{s1}

Distributed optimization has received immense attention in the past decade, mostly inspired by advanced and inexpensive sensors, big data, and large-scale networks, and so on. In distributed optimization, a network consisting of a family of agents is usually introduced to capture the communication pattern among all agents, where each agent is only accessible to partial (and maybe private) information on the global optimization problem. In this case, the agents in the network aim to cooperatively, by local information exchange, solve the global optimization problem.

To date, a large volume of algorithms have been devised for distributed optimization problems. Generally speaking, the existing algorithms can be roughly summarized as two classes: consensus-based algorithms and dual-decomposition-based algorithms. Wherein, consensus-based algorithms employ the consensus idea to align the estimated variables of all agents, for which existing algorithms include distributed subgradient \cite{nedic2009distributed}, diffusion adaptation strategy \cite{chen2012diffusion}, fast distributed gradient \cite{jakovetic2014fast}, asynchronous distributed gradient \cite{xu2018convergence}, stochastic mirror descent \cite{yuan2018optimal}, and distributed quasi-monotone subgradient algorithm \cite{liang2019distributed}, etc. With regard to dual-decomposed-based algorithms, dual variables are usually introduced by viewing the synchronization of all local variables as equality constraints, including alternating direction method of multipliers (ADMM) \cite{shi2014linear}, EXTRA \cite{shi2015extra}, augmented Lagrangian method \cite{jakovetic2015linear}, distributed dual proximal gradient \cite{notarnicola2017asynchronous}, and distributed forward-backward Bregman splitting \cite{xu2018bregman}.

From another viewpoint, a variety of scenarios have so far been considered for distributed optimization. The simplest case is to minimize an objective/cost function without any constraints \cite{nedic2009distributed,shi2015extra,jerinkic2019distributed}, including feasible set constraints, equality and inequality constraints, where the objective function is separable and composed of local objective functions. A little more complex case is to address distributed optimization with global/local feasible set constraints \cite{nedic2010constrained,lin2016distributed2,liu2017convergence}, that is, the decision variable must stay within some pre-specified nonempty set that is often assumed to be closed and convex. Moreover, the scenario with local (affine) equality constraints are addressed, for example, in \cite{liu2017constrained}, while local inequality constraints are investigated such as in \cite{yang2016multi}, and global inequality constraints that can be realized by all agents are taken into account in the literature, see \cite{zhu2012distributed} for an example. Furthermore, the case with globally coupled inequality constraints, where individual agent is only capable of accessing partial information on the global inequality constraints, is studied such as in \cite{chang2014distributed,mateos2017distributed,falsone2017dual,notarnicola2017duality,notarnicola2019constraint,li2018distributedon}, and meanwhile, time-varying objective functions and/or constraint functions are also considered in recent years \cite{lee2017sublinear,li2019distributed4,yi2019distributed,yi2019distributed2}.

With careful observation, it can be found that distributed optimization studied in the aforementioned works focus on the case where a global objective function is a sum of local objective functions, which are dependent only on their own decision variables. To be specific, the problem is in the form $\sum_{i=1}^N f_i(x_i)$ such that $x_i=x_j$ for all $i\neq j$, maybe subject to inequality constraints, from which it is easy to see that each $f_i$ is a function with respect to only $x_i$, independent of any other variables $x_j,j\neq i$. However, in a multitude of practical applications, local objective functions are also determined by other agents' variables. For example, in multi-agent formation control, each objective function often relies on variables (such as positions or velocities) of all its neighbors, and this scenario has been considered such as in \cite{cao2020distributed} and \cite{li2019distributed-ijrnc} (cf. Remark 4). As another example, the average of all variables, i.e., $\sum_{i=1}^N x_i/N$, is a vital parameter for all agents in a network, which can be discovered from a large number of applications, such as optimal placement problem, transportation network, and formation control, etc. For instance, in formation control, a group of networked agents desire to achieve a geometric pattern, and simultaneously, they may plan to encircle an important target, which can be cast as a target tracking problem for the center of all agents. Therefore, it is significant to deal with the scenario where the average of all variables is involved in local objective functions. From the theoretical perspective, when each local function $f_i$ also depends on variables of other agents (such as the average $\sum_{i=1}^N x_i/N$), the problem will be more challenging since other variables (such as the average $\sum_{i=1}^N x_i/N$) and related gradients are unavailable to agent $i$.

Motivated by the above facts, this paper aims to formulate and study a new framework for distributed optimization, called distributed aggregative optimization, for which a distributed algorithm, called distributed gradient tracking (DGT), is developed and analyzed. It is shown that the proposed algorithm has a linear convergence speed under mild assumptions, such as strong convexity of the global objective function and a directed balanced communication graph. The contributions of this paper are as follows: (1) a new distributed aggregative optimization is formulated for the first time; (2) a linearly convergent distributed algorithm is proposed and analyzed rigorously; and (3) a numerical example is provided to support the theoretical result.

The rest of this paper is structured as follows. Some preliminaries and the problem formulation are provided in Section \ref{s2}, followed by the main result in Section \ref{s3}. In Section \ref{s4}, a numerical example is presented to corroborate the theoretical result, and the conclusion is drawn in Section \ref{s5}.

{\em Notations:} Let $\mathbb{R}^n$ and $\mathbb{C}$ be the set of vectors with dimension $n>0$ and the set of complex numbers, respectively. Define $[k]=\{1,2,\ldots,k\}$ for an integer $k>0$. Denote by $col(z_1,\ldots,z_k)$ the column vector by stacking up $z_1,\ldots,z_k$. Let $\|\cdot\|$, $x^\top$, and $\langle x,y\rangle$ be the standard Euclidean norm, the transpose of $x\in\mathbb{R}^n$, and standard inner product of $x,y\in\mathbb{R}^n$. Let $\mathbf{1}$ and $\mathbf{0}$ be column vectors of compatible dimension with all entries being $1$ and $0$, respectively, and $I$ be the compatible identity matrix. Let $\rho(M)$ be the spectral radius of a square matrix $M$. $\otimes$ is the Kronecker product. Let $J:=\frac{1}{N}\bf{1}\bf{1}^\top$ and $\mathcal{J}:=J\otimes I$ with compatible dimension.

\section{Preliminaries}\label{s2}

\subsection{Graph Theory}\label{s2.1}

The communication pattern among all agents is captured by a simple graph in this paper, denoted by $\mathcal{G}=(\mathcal{V},\mathcal{E})$ with the node set $\mathcal{V}=\{1,\ldots,N\}$ and the edge set $\mathcal{E}\subset\mathcal{V}\times\mathcal{V}$. An edge $(j,i)\in\mathcal{E}$ means that node $j$ can send information to node $i$, where $j$ is called an in-neighbor of $i$. Denote by $\mathcal{N}_i=\{j:(j,i)\in\mathcal{E}\}$ the in-neighbor set of node $i$. The graph $\mathcal{G}$ is called undirected if $(i,j)\in\mathcal{E}$ is equivalent to $(j,i)\in\mathcal{E}$, and directed otherwise. The communication matrix $A=(a_{ij})\in\mathbb{R}^{N\times N}$ is defined by: $a_{ij}>0$ if $(j,i)\in\mathcal{E}$, and $a_{ij}=0$ otherwise.

The following standard assumptions on the communication graph are postulated.

\begin{assumption}\label{a1}
The following hold for the interaction graph:
\begin{enumerate}
  \item The graph $\mathcal{G}$ is strongly connected;
  \item The matrix $A$ is doubly stochastic, i.e., $\sum_{j=1}^N a_{ij}=1$ and $\sum_{i=1}^N a_{ij}=1$ for all $i,j\in[N]$.
\end{enumerate}
\end{assumption}

It should be noted that some approaches have been brought forward in the literature in order to hold the double stochasticity condition, for example, the uniform weights \cite{blondel2005convergence} and the least-mean-square consensus weight rules \cite{xiao2007distributed}. In addition, some distributed strategies have been proposed in \cite{gharesifard2012distributed} for strongly connected directed graphs to compute a doubly stochastic assignment in finite time.

\subsection{Convex Optimization}\label{s2.2}


For a convex function $g:\mathbb{R}^n\to \mathbb{R}$, the {\em subdifferential}, denoted as $\partial g(x)$, of $g$ at $x$ is defined by
\begin{align}
\partial g(x)=\{s\in\mathbb{R}^n:~g(y)-g(x)\geq s^\top(y-x),~\forall y\in\mathbb{R}^n\},        \nonumber
\end{align}
and each element in $\partial g(x)$ is called a {\em subgradient}. When $g$ is differentiable at $x$, the subdifferential $\partial g(x)$ only contains one element, which is usually called {\em gradient}, denoted as $\nabla g(x)$.

The differentiable function $g$ is called {\em $\mu$-strongly convex} if for all $x,y\in\mathbb{R}^n$,
\begin{align}
g(x)\geq g(y)+\nabla g(y)^\top (x-y)+\frac{\mu}{2}\|x-y\|^2.        \label{3}
\end{align}

\subsection{Problem Formulation}\label{s2.3}

This paper proposes a new framework for distributed optimization in a network composed of $N$ agents, called {\em distributed aggregative optimization}, given as follows:
\begin{align}
\min_{x\in \mathbb{R}^n}~~f(x)&:=\sum_{i=1}^N f_i(x_i,\sigma(x)),                \label{5}\\
\sigma(x)&:=\frac{\sum_{i=1}^N \phi_i(x_i)}{N},                     \label{6}
\end{align}
where $x=col(x_1,\ldots,x_N)$ is the global decision variable with $x_i\in\mathbb{R}^{n_i}$, $n:=\sum_{i=1}^N n_i$, and $f_i:\mathbb{R}^n\to\mathbb{R}$ is the local objective function. In problem (\ref{5}), the global function $f$ is not known to any agent, and each agent can only privately access the information on $f_i$. Moreover, each agent $i\in[N]$ is only aware of the decision variable $x_i$ without any knowledge of $x_j$'s for all $j\neq i$. Moreover, the term $\sigma(x)$ is an aggregative information of all agents' variables, and the function $\phi_i:\mathbb{R}^{n_i}\to\mathbb{R}^d$ is only accessible to agent $i$. The goal is to design distributed algorithms to seek an optimal decision variable for problem (\ref{5}).

\begin{remark}\label{r1}
It should be noted that distributed aggregative optimization is proposed here for the first time, to our best knowledge, which is different from aggregative games \cite{liang2017distributed2,de2019continuous}. The substantial difference lies in that all agents in problem (\ref{5}) aim to cooperatively find an optimal variable for the sum of all local objective functions, while the objective of aggregative games is to find the Nash equilibrium in a noncooperative manner since each agent desires to minimize only its own objective function. This can be seen from the following simple example.
\end{remark}

\begin{example}\label{e1}
As a simple example, let us consider two agents in a network in the scalar space $\mathbb{R}$ without feasible set constraints. Let $f_1(x)=(x_1-1)^2+\sigma^2(x)=(x_1-1)^2+(x_1+x_2)^2/4$ and $f_2(x)=(x_2-2)^2+\sigma^2(x)=(x_2-2)^2+(x_1+x_2)^2/4$. As a result, for distributed aggregative optimization, the optimal variable of $f(x)=f_1(x)+f_2(x)$ can be easily calculated, by $\nabla_{x_1}f(x)=0$ and $\nabla_{x_2}f(x)=0$, as $x_1=1/4,x_2=5/4$. On the other hand, as for aggregative games, the Nash equilibrium can be computed, by $\nabla_{x_1}f_1(x)=0$ and $\nabla_{x_2}f_2(x)=0$, as $x_1=1/2,x_2=3/2$. It is apparent to see that the Nash equilibrium $x_1=1/2,x_2=3/2$ is not the same as the global optimizer of $f(x)$, i.e., $x_1=1/4,x_2=5/4$. In other words, the Nash equilibrium is generally not the optimal decision variable due to the noncooperative nature of all agents in aggregative games.
\end{example}

To move forward, for brevity, let $\nabla_1 f_i(x_i,\sigma(x))$ and $\nabla_2 f_i(x_i,\sigma(x))$ denote $\nabla_{x_i}f_i(x_i,\sigma(x))$ and $\nabla_{\sigma}f_i(x_i,\sigma(x))$, respectively, for all $i\in[N]$. And for $x\in\mathbb{R}^n$ and $y=col(y_1,\ldots,y_N)\in\mathbb{R}^{Nd}$, define $f(x,y):=\sum_{i=1}^N f_i(x_i,y_i)$, $\nabla_1 f(x,y):=col(\nabla_1 f_1(x_1,y_1),\ldots,\nabla_1 f_N(x_N,y_N))$ and $\nabla_2 f(x,y):=col(\nabla_2 f_1(x_1,y_1),\ldots,\nabla_2 f_N(x_N,y_N))$.

It is now necessary to list some assumptions.

\begin{assumption}\label{a2}
The following hold for problem (\ref{5}):
\begin{enumerate}
  \item The global objective function $f(x)$ is differentiable, $\mu$-strongly convex, and $L_1$-smooth on $\mathbb{R}^n$, that is, $\|\nabla f(x)-\nabla f(x')\|\leq L_1\|x-x'\|$ for all $x,x'\in \mathbb{R}^n$. Also, $\nabla_1 f(x,y)+\nabla\phi(x){\bf 1}_N\otimes\frac{1}{N}\sum_{i=1}^N\nabla_2 f_i(x_i,y_i)$ is $L_1$-Lipschitz;
  \item $\nabla_2 f(x,y)$ is $L_2$-Lipschitz continuous, that is, $\|\nabla_2 f(x,y)-\nabla_2 f(x',y')\|\leq L_2(\|x-x'\|+\|y-y'\|)$ for all $x,x'\in\mathbb{R}^n$ and $y,y'\in\mathbb{R}^{Nd}$;
  \item All $\phi_i$'s are differentiable, and there exists a constant $L_3>0$ such that $\|\nabla \phi_i(x_i)\|\leq L_3$ for all $x_i\in\mathbb{R}^{n_i}$ and $i\in[N]$.
\end{enumerate}
\end{assumption}
It should be noted that the Lipschitz property of $\nabla f(x)$ and $\nabla_1 f(x,y)+\nabla\phi(x){\bf 1}_N\otimes\frac{1}{N}\sum_{i=1}^N\nabla_2 f_i(x_i,y_i)$ in Assumption \ref{a2}.1 can be ensured by Assumptions \ref{a2}.2 and \ref{a2}.3 along with the boundedness of $\nabla_2 f(x,y)$ and the Lipschitz property of $\nabla_1 f(x,y)$ and $\nabla\phi_i(x_i)$, which are standard in distributed optimization and game theory (e.g., \cite{nedic2009distributed,jakovetic2014fast,liu2017convergence,chang2014distributed,li2018distributedon,liang2017distributed2,de2019continuous}). Please also note that it is only assumed the strong convexity of the global objective function $f(x)$, without even the convexity of local objective functions $f_i$'s.

To conclude this section, it is useful to display a few lemmas.


\begin{lemma}[\cite{horn2012matrix}]\label{l1}
For an irreducible nonnegative matrix $M\in\mathbb{R}^{n\times n}$, it is primitive if it has at least one non-zero diagonal entry.
\end{lemma}

\begin{lemma}[\cite{horn2012matrix}]\label{l2}
For an irreducible nonnegative matrix $M\in\mathbb{R}^{n\times n}$, there hold (i) $\rho(M)>0$ is an eigenvalue of $M$, (ii) $Mx=\rho(M)x$ for some positive vector $x$, and (iii) $\rho(M)$ is an algebraically simple eigenvalue.
\end{lemma}

\begin{lemma}\label{l3}
Let $F:\mathbb{R}^n\to\mathbb{R}$ be $\mu$-strongly convex and $L$-smooth. Then $\|x-\alpha\nabla F(x)-(y-\alpha\nabla F(y))\|\leq (1-\mu\alpha)\|x-y\|$ for all $x,y\in\mathbb{R}^n$, where $\alpha\in(0,1/L]$.
\end{lemma}
\begin{proof}
It is known that a convex function $f:\mathbb{R}^n\to\mathbb{R}$ is $l$-smooth is equivalent to the convexity of $\frac{l}{2}\|x\|^2-f(x)$. Thus, by $L$-smoothness of $F$, one has that $\frac{L}{2}\|x\|^2-F(x)$ is convex, and then $\frac{1}{2\alpha}\|x\|^2-F(x)$ is $(\frac{1}{2\alpha}-\frac{L}{2})$-strongly convex for $\alpha\in(0,1/L]$, which further implies that $H(x):=\frac{1}{2}\|x\|^2-\alpha F(x)$ is $(\frac{1}{2}-\frac{L\alpha}{2})$-strongly convex. Meanwhile, it is easy to verify that $\frac{1-\mu\alpha}{2}\|x\|^2-H(x)=\alpha(F(x)-\frac{\mu}{2}\|x\|^2)$ is convex since $F(x)-\frac{\mu}{2}\|x\|^2$ is convex due to the $\mu$-strong convexity of $F$. Therefore, $H$ is $(1-\mu\alpha)$-smooth, i.e., $\|\nabla H(x)-\nabla H(y)\|\leq (1-\mu\alpha)\|x-y\|$ for all $x,y\in\mathbb{R}^n$, thus ending the proof.
\end{proof}

\begin{lemma}\label{l4}
Under Assumption \ref{a1}, there hold (i) $\mathcal{A}\mathcal{J}=\mathcal{J}\mathcal{A}=\mathcal{J}$, and (ii) $\|\mathcal{A}x-\mathcal{J}x\|\leq \rho\|x-\mathcal{J}x\|$ for any $x\in\mathbb{R}^{Nd}$, where $\mathcal{A}:=A\otimes I_d$ and $\rho:=\|A-J\|<1$.
\end{lemma}
\begin{proof}
The assertion (i) is trivial to verify. For the assertion (ii), it is easy to see that $\|\mathcal{A}x-\mathcal{J}x\|=\|(\mathcal{A}-\mathcal{J})x-(\mathcal{A}-\mathcal{J})\mathcal{J}x\|\leq \|\mathcal{A}-\mathcal{J}\|\|x-\mathcal{J}x\|=\|A-J\|\|x-\mathcal{J}x\|$. Invoking the double-stochasticity of $A$ and the Perron-Frobenius theorem \cite{horn2012matrix}, one has $\|A-J\|<1$. This ends the proof.
\end{proof}

\begin{lemma}\label{l5}\cite{xin2018linear}
Let $X,E\in\mathbb{R}^{n\times n}$ with $\lambda$ being a simple eigenvalue of $X$. Let $w$ and $v$ be the left and right eigenvectors of $X$ associated with the eigenvalue $\lambda$, respectively. Then,
\begin{enumerate}
  \item for each $\epsilon >0$, there exists a $\delta>0$ such that, $\forall t\in\mathbb{C}$ with $|t|<\delta$, there is a unique eigenvalue $\lambda(t)$ of $X+tE$ such that $|\lambda(t)-\lambda-t\frac{w^\top E v}{w^\top v}|\leq |t|\epsilon$,
  \item $\lambda(t)$ is continuous at $t=0$, and $\lim_{t\to 0}\lambda(t)=\lambda$,
  \item $\lambda(t)$ is differentiable at $t=0$, and $\frac{d\lambda(t)}{dt}\big|_{t=0}=\frac{w^\top E v}{w^\top v}$.
\end{enumerate}
\end{lemma}

\section{Main Result}\label{s3}

This section presents the algorithm design and analysis. In doing so, a distributed algorithm, called {\em distributed gradient tracking} (DGT for short), for solving (\ref{5}) is proposed for each agent $i\in[N]$ as in Algorithm 1.

\begin{algorithm}
 \caption{Distributed Gradient Tracking (DGT)}
 \begin{algorithmic}[1]
  \STATE \textbf{Initialization:} Stepsize $\alpha>0$, and initial conditions $x_{i,0}\in \mathbb{R}^{n_i}$, $\sigma_{i,0}=\phi_i(x_{i,0})$, and $y_{i,0}=\nabla_2 f_{i}(x_{i,0},\sigma_{i,0})$ for all $i\in[N]$.
  \STATE \textbf{Iterations:} Step $k\geq 0$: update for each $i\in[N]$:
\begin{subequations}
\begin{align}
x_{i,k+1}&=x_{i,k}-\alpha[\nabla_1 f_i(x_{i,k},\sigma_{i,k})+\nabla\phi_i(x_{i,k})y_{i,k}],      \label{7.1}\\
\sigma_{i,k+1}&=\sum_{j=1}^N a_{ij}\sigma_{j,k}+\phi_i(x_{i,k+1})-\phi_i(x_{i,k}),                                     \label{7.2}\\
y_{i,k+1}&=\sum_{j=1}^N a_{ij}y_{j,k}+\nabla_2 f_i(x_{i,k+1},\sigma_{i,k+1})            \nonumber\\
&\hspace{2.5cm}-\nabla_2 f_i(x_{i,k},\sigma_{i,k}),       \label{7.3}
\end{align}              \label{7}
\end{subequations}
 \end{algorithmic}
\end{algorithm}

In algorithm (\ref{7}), $\sigma_{i,k}$ is leveraged for agent $i$ to track the average (\ref{6}) since $\sigma(x)$ is global information, which cannot be accessed directly for all agents, and meanwhile, $y_{i,k}$ is introduced for agent $i$ to track the gradient sum $\frac{1}{N}\sum_{i=1}^N \nabla_2 f_i(x_{i},\sigma(x))$, which is also unavailable to all agents. The initial variable $x_{i,0}$ is arbitrary for all $i\in[N]$, and choosing $\sigma_{i,0}=\phi_i(x_{i,0})$ and $y_{i,0}=\nabla_2 f_i(x_{i,0},\sigma_{i,0})$ for all $i\in[N]$.

The name ``distributed gradient tracking'' is attributed to the fact that algorithm (\ref{7}) has combined the classical gradient descent algorithm with the variable tracking techniques.

To proceed, for a vector $x=col(x_1,\ldots,x_N)\in\mathbb{R}^n$, it is helpful to define $\phi(x):=col(\phi_1(x_1),\ldots,\phi_N(x_N))$. Also, for a differentiable function $g(x)=col(g_1(x),\ldots,g_m(x))$, where $g_i$'s are real-valued functions, let us denote by $\nabla g(x)=(\nabla g_1(x),\ldots,\nabla g_m(x))$.

With the above notations and those after Example \ref{e1}, DGT (\ref{7}) can be written in a compact form
\begin{align}
x_{k+1}&=x_{k}-\alpha[\nabla_1 f(x_{k},\sigma_{k})+\nabla\phi(x_{k})y_{k}],           \label{8}\\
\sigma_{k+1}&=\mathcal{A}\sigma_k+\phi(x_{k+1})-\phi(x_{k}),                                 \label{9}\\
y_{k+1}&=\mathcal{A}y_k+\nabla_2 f(x_{k+1},\sigma_{k+1})-\nabla_2 f(x_{k},\sigma_{k}),       \label{10}
\end{align}
with $\mathcal{A}=A\otimes I_d$ as defined in Lemma \ref{l4}, $x_k:=col(x_{i,k},\ldots,x_{N,k})$, and similar notations for $\sigma_k$ and $y_k$.

Before presenting the main result, it is necessary to first introduce a preliminary result.
\begin{lemma}\label{l5}
Under Assumption \ref{a1}, there hold:
\begin{align*}
\bar{\sigma}_k&:=\frac{1}{N}\sum_{i=1}^N\sigma_{i,k}=\frac{1}{N}\sum_{i=1}^N\phi_i(x_{i,k}),    \\
\bar{y}_k&:=\frac{1}{N}\sum_{i=1}^Ny_{i,k}=\frac{1}{N}\sum_{i=1}^N\nabla_2 f_i(x_{i,k},\sigma_{i,k}).
\end{align*}
\end{lemma}
\begin{proof}
In view of (\ref{9}) and double-stochasticity in Assumption \ref{a1}, multiplying ${\bf 1}^\top/N$ on both sides of (\ref{9}) can lead to that
\begin{align*}
\bar{\sigma}_{k+1}=\bar{\sigma}_k+\frac{1}{N}\sum_{i=1}^N\phi_i(x_{i,k+1})-\frac{1}{N}\sum_{i=1}^N\phi_i(x_{i,k}),
\end{align*}
which further implies that
\begin{align*}
\bar{\sigma}_{k}-\frac{1}{N}\sum_{i=1}^N\phi_i(x_{i,k})=\bar{\sigma}_0-\frac{1}{N}\sum_{i=1}^N\phi_i(x_{i,0}).
\end{align*}
Combining the above equality and $\sigma_{i,0}=\phi_i(x_{i,0})$ yields the first assertion of this lemma. Similar arguments can obtain the second one, which completes the proof.
\end{proof}

It is now ready to present the main result of this paper.

\begin{theorem}\label{t1}
Under Assumptions \ref{a1} and \ref{a2}, if
\begin{align}
0<\alpha< \min\Big\{\frac{1}{L_1},\alpha_s\Big\},            \label{t1a}
\end{align}
where
\begin{align}
\alpha_s:=\frac{\mu(1-\rho)^2}{L_3L_\mu[(1-\rho)L_0+2L_2L_3]},            \label{t1b}
\end{align}
$L_\mu:=\mu+L_1+L_2L_3$ and $L_0:=L_1+L_2+L_2L_3$, then $x_k=col(x_{1,k},\ldots,x_{N,k})$ generated by algorithm (\ref{7}) can converge to the optimizer of problem (\ref{5}) at a linear convergence rate.
\end{theorem}

\begin{proof}
Let us bound $\|x_{k+1}-x^*\|$, $\|x_{k+1}-x_k\|$, $\|\sigma_{k+1}-\mathcal{J}\sigma_{k+1}\|$, and $\|y_{k+1}-\mathcal{J}y_{k+1}\|$ in the sequel, where $x^*$ is the optimal variable of problem (\ref{5}). Denote $\sigma(x^*)$ as $\sigma^*$ for brevity in this proof.

First, for $\|x_{k+1}-x^*\|$, invoking (\ref{8}) yields that
\begin{align}
&\|x_{k+1}-x^*\|      \nonumber\\
&= \|x_{k}-x^*-\alpha[\nabla_1 f(x_{k},\sigma_{k})+\nabla\phi(x_{k})y_{k}]\|          \nonumber\\
&\leq \|x_{k}-x^*-\alpha[\nabla_1 f(x_{k},{\bf 1}_N\otimes\bar{\sigma}_{k})             \nonumber\\
&\hspace{0.2cm}+\nabla\phi(x_{k}){\bf 1}_N\otimes \frac{1}{N}\sum_{i=1}^N\nabla_2 f_i(x_{i,k},{\bf 1}_N\otimes\bar{\sigma}_k)]+\alpha\nabla f(x^*)          \nonumber\\
&\hspace{0.2cm}+\alpha\|\nabla_1 f(x_{k},\sigma_{k})+\nabla\phi(x_{k}){\bf 1}_N\otimes \bar{y}_{k}-\nabla_1 f(x_{k},{\bf 1}_N\otimes\bar{\sigma}_{k})   \nonumber\\
&\hspace{0.2cm}-\nabla\phi(x_{k}){\bf 1}_N\otimes \frac{1}{N}\sum_{i=1}^N\nabla_2 f_i(x_{i,k},{\bf 1}_N\otimes\bar{\sigma}_k)\|          \nonumber\\
&\hspace{0.2cm}+\alpha\|\nabla\phi(x_{k})y_{k}-\nabla\phi(x_{k}){\bf 1}_N\otimes \bar{y}_{k}\|           \nonumber\\
&\leq (1-\mu\alpha)\|x_k-x^*\|+\alpha L_1\|\sigma_k-{\bf 1}_N\otimes \bar{\sigma}_k\|               \nonumber\\
&\hspace{0.2cm}+\alpha\|\nabla\phi(x_k)\|\|y_k-{\bf 1}_N\otimes \bar{y}_k\|                   \nonumber\\
&\leq (1-\mu\alpha)\|x_k-x^*\|+\alpha L_1\|\sigma_k-\mathcal{J}\sigma_k\|                   \nonumber\\
&\hspace{0.2cm}+\alpha L_3\|y_k-\mathcal{J}y_k\|,                                          \label{10a}
\end{align}
where Assumption \ref{a2}.1, Lemma \ref{l3}, (\ref{8}) and (\ref{t1a}) have been utilized to obtain the second inequality, and the last inequality has applied the fact that $\|\nabla\phi(x_k)\|\leq \max_{i\in[N]}\|\phi_i(x_{i,k})\|\leq L_3$ by Assumption \ref{a2}.3, ${\bf 1}_N\otimes \bar{\sigma}_k=\mathcal{J}\sigma_k$, and ${\bf 1}_N\otimes \bar{y}_k=\mathcal{J}y_k$.

Second, for $\|x_{k+1}-x_k\|$, by noting that
\begin{align*}
\nabla f(x^*)&=\nabla_1 f(x^*,{\bf 1}_N\otimes\sigma^*)           \nonumber\\
&\hspace{0.4cm}+\nabla\phi(x^*)[{\bf 1}_N\otimes \frac{1}{N}\sum_{i=1}^N\nabla_2 f_i(x^*,{\bf 1}_N\otimes\sigma^*)]          \nonumber\\
&=\textbf{0},
\end{align*}
invoking (\ref{8}) yields that
\begin{align}
&\|x_{k+1}-x_k\|          \nonumber\\
&=\alpha\|\nabla_1f(x_k,\sigma_k)+\nabla\phi(x_k)y_k\|        \nonumber\\
&\leq \alpha\|\nabla_1f(x_k,\sigma_k)+\nabla\phi(x_k)\mathcal{J}y_k-\nabla_1 f(x^*,{\bf 1}_N\otimes\sigma^*)    \nonumber\\
&\hspace{0.3cm}-\nabla\phi(x^*)[{\bf 1}_N\otimes \frac{1}{N}\sum_{i=1}^N\nabla_2 f_i(x^*,{\bf 1}_N\otimes\sigma^*)]\|           \nonumber\\
&\hspace{0.3cm}+\alpha\|\nabla\phi(x_k)(y_k-\mathcal{J}y_k)\|           \nonumber\\
&\leq \alpha L_1(\|x_k-x^*\|+\|\sigma_k-{\bf 1}_N\otimes\sigma^*\|)+\alpha L_3\|y_k-\mathcal{J}y_k\|       \nonumber\\
&\leq \alpha L_1(\|x_k-x^*\|+\|\sigma_k-\mathcal{J}\sigma_k\|)+\alpha L_3\|y_k-\mathcal{J}y_k\|    \nonumber\\
&\hspace{0.3cm}+\alpha L_1\|\mathcal{J}\sigma_k-{\bf 1}_N\otimes\sigma^*\|,                      \label{11}
\end{align}
where Assumption \ref{a2}.1 and $\|\nabla\phi(x_k)\|\leq L_3$ have been used in the second inequality. For the last term in (\ref{11}), in view of Lemma \ref{l5}, one has that
\begin{align*}
\|\mathcal{J}\sigma_k-{\bf 1}_N\otimes\sigma^*\|^2&=\|{\bf 1}_N\otimes (\bar{\sigma}_k-\sigma^*)\|^2                   \nonumber\\
&=N\|\frac{1}{N}\sum_{i=1}^N(\phi_i(x_{i,k})-\phi_i(x_i^*))\|^2       \nonumber\\
&\leq \frac{1}{N}(\sum_{i=1}^N\|\phi_i(x_{i,k})-\phi_i(x_i^*)\|)^2     \nonumber\\
&\leq \frac{1}{N}(\sum_{i=1}^N L_3\|x_{i,k}-x_i^*\|)^2                 \nonumber\\
&\leq L_3^2\sum_{i=1}^N\|x_{i,k}-x_i^*\|^2                             \nonumber\\
&=L_3^2\|x_k-x^*\|^2,
\end{align*}
where Assumption \ref{a2}.3 has been employed in the second inequality, and the last inequality has appealed to the fact that $(\sum_{i=1}^N a_i)^2\leq N\sum_{i=1}^N a_i^2$ for any nonnegative scalars $a_i$'s. Therefore, combining the above inequality and (\ref{11}) follows that
\begin{align}
\|x_{k+1}-x_k\|&\leq\alpha L_1(1+L_3)\|x_k-x^*\|+\alpha L_1\|\sigma_k-\mathcal{J}\sigma_k\|           \nonumber\\
&\hspace{0.4cm}+\alpha L_3\|y_k-\mathcal{J}y_k\|.                      \label{12}
\end{align}

Third, regarding $\|\sigma_{k+1}-\mathcal{J}\sigma_{k+1}\|$, by noting that $\mathcal{J}\mathcal{A}=\mathcal{A}\mathcal{J}=\mathcal{J}$, in light of (\ref{9}), one can obtain that
\begin{align}
&\|\sigma_{k+1}-\mathcal{J}\sigma_{k+1}\|   \nonumber\\
&=\|\mathcal{A}\sigma_k+\phi(x_{k+1})-\phi(x_k)-\mathcal{J}\mathcal{A}\sigma_k       \nonumber\\
&\hspace{0.3cm}-\mathcal{J}[\phi(x_{k+1})-\phi(x_k)]\|          \nonumber\\
&\leq \rho\|\sigma_k-\mathcal{J}\sigma_k\|+\|I-\mathcal{J}\|\|\phi(x_{k+1})-\phi(x_k)\|    \nonumber\\
&\leq \rho\|\sigma_k-\mathcal{J}\sigma_k\|+L_3\|I-\mathcal{J}\|\|x_{k+1}-x_k\|,             \label{13}
\end{align}
where Lemma \ref{l4} has been leveraged in the first inequality, and Assumption \ref{a2}.3 has been exploited in the last inequality. By noticing that $\|I-\mathcal{J}\|=1$ and inserting (\ref{12}) into (\ref{13}), it can be obtained that
\begin{align}
&\|\sigma_{k+1}-\mathcal{J}\sigma_{k+1}\|   \nonumber\\
&\leq (\rho+\alpha L_1L_3)\|\sigma_k-\mathcal{J}\sigma_k\|+\alpha L_1 L_3(1+L_3)\|x_k-x^*\|       \nonumber\\
&\hspace{0.4cm}+\alpha L_3^2\|y_k-\mathcal{J}y_k\|.             \label{14}
\end{align}

Fourth, for $\|y_{k+1}-\mathcal{J}y_{k+1}\|$, similar to (\ref{13}), invoking (\ref{10}) results in that
\begin{align}
&\|y_{k+1}-\mathcal{J}y_{k+1}\|              \nonumber\\
&\leq \rho\|y_k-\mathcal{J}y_k\|+\|\nabla_2 f(x_{k+1},\sigma_{k+1})-\nabla_2 f(x_k,\sigma_k)\|.        \label{15}
\end{align}

At this step, by (\ref{9}), one has that $\|\sigma_{k+1}-\sigma_k\|=\|(\mathcal{A}-I\otimes I_d)(\sigma_k-\mathcal{J}\sigma_k)+\phi(x_{k+1})-\phi(x_k)\|\leq \|A-I\|\|\sigma_k-\mathcal{J}\sigma_k\|+L_3\|x_{k+1}-x_k\|$, where the fact $\mathcal{A}\sigma_k-\sigma_k=(\mathcal{A}-I\otimes I_d)(\sigma_k-\mathcal{J}\sigma_k)$ has been used in the equality, and Assumption \ref{a2}.3 has been leveraged in the inequality, which together with Assumption \ref{a2}.2 yields that
\begin{align}
&\hspace{-0.2cm}\|\nabla_2 f(x_{k+1},\sigma_{k+1})-\nabla_2 f(x_k,\sigma_k)\|         \nonumber\\
&\hspace{-0.2cm}\leq L_3(\|x_{k+1}-x_k\|+\|\sigma_{k+1}-\sigma_k\|)                \nonumber\\
&\hspace{-0.2cm}\leq L_2(1+L_3)\|x_{k+1}-x_k\|+L_2\|A-I\|\|\sigma_k-\mathcal{J}\sigma_k\|.         \label{16}
\end{align}

Substituting (\ref{12}) and (\ref{16}) into (\ref{15}) can give rise to
\begin{align}
&\|y_{k+1}-\mathcal{J}y_{k+1}\|       \nonumber\\
&\leq (\rho+\alpha L_2L_3(1+L_3))\|y_k-\mathcal{J}y_k\|        \nonumber\\
&\hspace{0.4cm}+\alpha L_1 L_2(1+L_3)^2\|x_k-x^*\|          \nonumber\\
&\hspace{0.4cm}+(\alpha L_1L_2(1+L_3)+L_2\|A-I\|)\|\sigma_k-\mathcal{J}\sigma_k\|.        \label{17}
\end{align}

Finally, define $\theta_k:=col(\|x_k-x^*\|,\|\sigma_k-\mathcal{J}\sigma_k\|,\|y_k-\mathcal{J}y_k\|)$. By (\ref{10a}), (\ref{14}), (\ref{17}) and $\|A-I\|\leq 2$, it can be concluded that
\begin{align}
\theta_{k+1}\leq M(\alpha) \theta_k,        \label{18}
\end{align}
where
\begin{align}
M(\alpha):=X+\alpha E,         \label{mat1}
\end{align}
and
\begin{align}
X:=\left(
     \begin{array}{ccc}
       1 & 0 & 0 \\
       0 & \rho & 0 \\
       0 & 2L_2 & \rho \\
     \end{array}
   \right),                  \label{mat2}
\end{align}
\begin{align}
E:=\left(
     \begin{array}{ccc}
       -\mu & L_1 & L_3 \\
       L_1L_2(1+L_3) & L_1L_3 & L_3^2 \\
       L_1L_2(1+L_3)^2 & L_1L_2(1+L_3) & L_2L_3(1+L_3) \\
     \end{array}
   \right).                \label{mat3}
\end{align}

Denote by $\lambda(\alpha)$ the eigenvalues of $M(\alpha)$. It is easy to see that $1$ is a simple eigenvalue of $M(0)$, and its corresponding left and right eigenvectors are both $w=col(1,0,0)$. As a result, invoking Lemma \ref{l5} leads to that
\begin{align}
\frac{d\lambda(\alpha)}{d\alpha}\Big|_{\alpha=0}=\frac{w^\top E w}{w^\top w}=-\mu<0,      \label{mat4}
\end{align}
which indicates that the spectral radius of $M(\alpha)$ will be less than $1$ for sufficiently small positive $\alpha$.

One can also see that the graph corresponding to $M(\alpha)$ is strongly connected, which together with Theorem C.3 in \cite{ren2008distributed} implies that $M(\alpha)$ is irreducible. By Lemma \ref{l1}, $M(\alpha)$ is primitive, which in combination with Lemma \ref{l2} can ensure that $1$ will be a simple eigenvalue of $M(\alpha)$ when $\alpha$ increases from $0$ to some value. By calculating $det(I-M(\alpha))=0$, one can obtain that $\alpha=\alpha_s$, where $\alpha_s$ is defined in (\ref{t1b}). Therefore, all eigenvalues of $M(\alpha)$ have absolute values less than $1$ when $\alpha\in(0,\alpha_s)$, which can guarantee the linear convergence rate of $\theta_k$, thus ending the proof.
\end{proof}

\begin{remark}
It is worth mentioning that, to our best knowledge, this paper is the first to investigate problem (\ref{5}) in the presence of the aggregative term $\sigma(x)$, for which a linearly convergent distributed algorithm has been developed here.
\end{remark}

\section{A Numerical Example}\label{s4}

This section aims at presenting an optimal placement problem for supporting the designed algorithm. In an optimal placement problem in $\mathbb{R}^2$, there are $M$ entities that are located at fixed positions, and meanwhile, there are $N$ free entities, each of which are only privately aware of some of the fixed $M$ entities. The objective is to determine the optimal positions of $N$ free entities in order to minimize the sum of all (square) distances from each free entity to its corresponding fixed entities and the (square) distances from each agent to the center of all free entities. For example, the entities can represent warehouses, the links between each free entity and its associated fixed entities as well as the center of all free entities stand for the transportation routes, and the center of all free entities means a goods factory or a central warehouse. In this example, free entities are called agents.

\begin{figure}[H]
\centering
\includegraphics[width=1.4in]{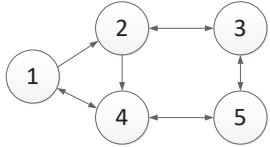}
\caption{The communication graph.}
\label{f1}
\end{figure}

For the above problem, let $M=N=5$, and each agent $i$ is only privately aware of the fixed entity $i$. In this case, the problem can be modeled as (\ref{5}) by letting
\begin{align}
f_i(x_i,\sigma(x))=\gamma_i\|x_i-r_i\|^2+\|x_i-\sigma(x)\|^2,~~i\in[N]       \label{19}
\end{align}
where $r_i$'s are the fixed entities, and $\gamma_i>0$ represents the weighting between the first and second terms. For the simulation, let $\phi_i$ be the identity mapping for all $i\in[N]$, $\alpha=0.05$, $\gamma_i=i$, $r_1=col(3,5)$, $r_2=col(6,9)$, $r_3=col(9,8)$, $r_4=col(6,2)$, and $r_5=col(9,2)$, and the communication graph is shown in Fig. \ref{f1}, which is strongly connected.

By randomly selecting the initial positions of agents, i.e., $x_{i,0}$'s, performing the developed DGT algorithm gives rise to evolutions of all $x_{i,k}$'s and $\sigma_{i,k}$'s, as shown in Figs. \ref{f2} and \ref{f3}, respectively, showing that all agents can converge to their optimal positions very fast and the estimate $\sigma_{i,k}$ of each agent can converge to the optimal $\sigma(x^*)$, where $x^*=col(x_1^*,\ldots,x_N^*)$ is the optimal position. Therefore, the simulation results support the theoretical result.

\begin{figure}[H]
\centering
\includegraphics[width=3.2in]{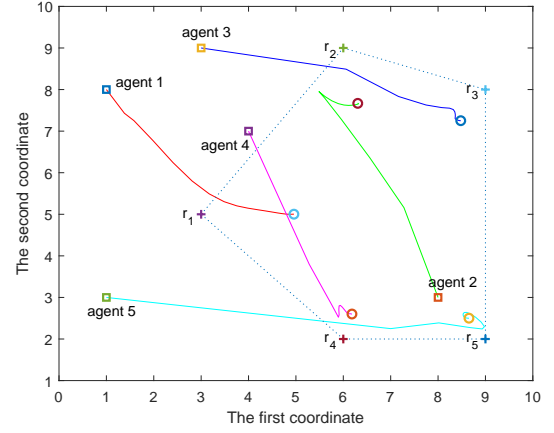}
\caption{Evolutions of $x_{i,k}$'s, where squares and circles mean initial positions and final optimal positions of all agents, respectively.}
\label{f2}
\end{figure}

\begin{figure}[H]
\centering
\includegraphics[width=3.2in]{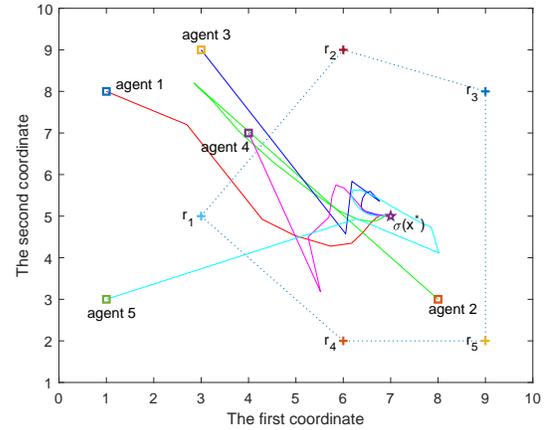}
\caption{Evolutions of $\sigma_{i,k}$'s, where squares are the initial positions, and the pentagram means the optimal center, i.e., $\sigma(x^*)=\frac{1}{N}\sum_{i=1}^N x_i^*$.}
\label{f3}
\end{figure}

\section{Conclusion}\label{s5}

This paper has proposed and investigated a new framework for distributed optimization, i.e., distributed aggregative optimization, which allows local objective functions to be dependent not only on their own decision variables but also on an aggregative term $\sigma(x)$, relying on decision variables of all other agents. To handle this problem, a distributed algorithm, i.e., DGT, has been developed and rigorously analyzed, where the global objective function is assumed to be strongly convex and smooth along with some Lipschitz property, and the communication graph is assumed to be fixed, balanced, and strongly connected. It has been shown that the algorithm can converge to the optimal variable at a linear rate. A numerical example has been provided to support the theoretical result. Basically, this paper opens up a new avenue to distributed optimization. Future works can be placed on various cases, such as unbalanced graphs, feasible constraint sets, and other interesting forms of objective functions, etc.

%
%
%



\end{document}